\documentclass[11pt]{amsart}%
\usepackage{color}
\usepackage{amsmath}
\usepackage{amssymb}
\usepackage{amsfonts}
\usepackage[latin1]{inputenc}
\usepackage{graphicx}%
\providecommand{\U}[1]{\protect\rule{.1in}{.1in}}
\providecommand{\U}[1]{\protect\rule{.1in}{.1in}}

\DeclareMathSymbol{\subsetneqq}{\mathbin}{AMSb}{36}

\theoremstyle{plain}
\numberwithin{equation}{section}
\newtheorem{theorem}{Theorem}[section]
\newtheorem{corollary}{Corollary}[section]
\newtheorem{lemma}{Lemma}[section]

\begin{document}
\title[Viscosity approximation method for a variational problem]{Viscosity approximation method for a variational problem}
\author{Ramzi May}%
\address{Mathematics Department, College of Science, King Faisal University, P.O. 380, Ahsaa 31982, Kingdom of Saudi Arabia}
\email{rmay@kfu.edu.sa}
\subjclass{47H09;47j05;47J25}
\keywords{Hilbert space; Variational inequality problem; nonexpansive mapping;
Inverse strongly monotone mappings}
\date{July 02, 2022}
\maketitle
\begin{abstract}
Let $Q$ be a nonempty closed and convex subset of a real Hilbert space $%
\mathcal{H}$, $S:Q\rightarrow Q$ a nonexpansive mapping, $A:Q\rightarrow Q$
an inverse strongly monotone operator, and $f:Q\rightarrow Q$ a contraction
mapping. We prove, under appropriate conditions on the real sequences $%
\{\alpha _{n}\}$ and $\{\lambda _{n}\},$ that for any starting point $x_{1}$
in $Q,$ the sequence $\{x_{n}\}$ generated by the iterative process
\begin{equation}
x_{n+1}=\alpha_{n}f(x_{n})+(1-\alpha_{n})SP_{Q}(x_{n}-\lambda_{n}Ax_{n})
\label{Alg}
\end{equation}
converges strongly to a particular element of the set $F_{ix}(S)\cap S_{VI(A,Q)}$ which we suppose that it is nonempty, where $F_{ix}(S)$ is the set of fixed point of the mapping $%
S$ and $S_{VI(A,Q)}$ is the set of $q\in Q$ such that $\langle
Aq,x-q\rangle\geq0$ for every $x\in Q.$ Moreover, we study the strong
convergence of a perturbed version of the algorithm generated by the above
process. Finally, we apply the main result to construct an algorithm
associated to a constrained convex optimization problem and we provide a
numerical experiment to emphasize the effect of the parameter $\{\alpha _{n}\}$
on the convergence rate of this algorithm.
\end{abstract}
\section{Introduction}

\noindent Let $\mathcal{H}$ be a real Hilbert space with inner product $%
\langle .,.\rangle $ and associated norm $\left\Vert .\right\Vert $.
Throughout this paper, we assume the following assumptions:
\par\noindent(A1) $Q$ is a nonempty, closed and convex subset of $\mathcal{H}$.
\par\noindent(A2) $S:Q\rightarrow Q$ is a nonexpansive mapping which means that $%
\left\Vert Sx-Sy\right\Vert \leq\left\Vert x-y\right\Vert $ for every $%
x,y\in Q$.
\par\noindent (A3) $A:Q\rightarrow \mathcal{H}$ is a $\nu$ inverse strongly
monotone operator which means that there exists a real $\nu>0$ such that
\begin{equation*}
\langle Ax-Ay,x-y\rangle\geq\nu\left\Vert Ax-Ay\right\Vert ^{2}~\forall
x,y\in Q;
\end{equation*}
(A4) $f:Q\rightarrow Q$ is a contraction with coefficient $\rho\in
\lbrack0,1[,$ that is
\begin{equation*}
\left\Vert f(x)-f(y)\right\Vert \leq\rho\left\Vert x-y\right\Vert ~\forall
x,y\in Q.
\end{equation*}
We denote by $F_{ix}(S)=\{x\in Q:Sx=x\}$ the set of fixed points of the
operator $S$ and by $S_{VI(A,Q)}$ the set of solutions of the following
variational inequality%
\begin{equation}
\text{Find }q\in Q\text{ such that }\langle Aq,x-q\rangle\geq0\text{ for
every }x\in Q.
\end{equation}
It is easy to prove that the sets $F_{ix}(S)$ and $S_{VI(A,Q)}$ are closed
and convex subsets of $\mathcal{H}$ (see Lemma \ref{L} and Lemma \ref{L2} in
the next section); hence the set
\begin{equation*}
\Omega:=F_{ix}(S)\cap S_{VI(A,Q)}
\end{equation*}
is also closed and convex subset of $\mathcal{H}.$ Hereafter, we assume
moreover that:
\par\noindent (A5): The set $\Omega$ is nonempty.

In this paper, we are interested in the numerical approximation of some
particular elements of $\Omega $. We recall that Takahashi and Toyoda \cite%
{TT} introduced the following algorithm%
\begin{equation}
\left\{
\begin{array}{l}
x_{1}\in Q \\
x_{n+1}=\alpha _{n}x_{n}+(1-\alpha _{n})SP_{Q}(x_{n}-\lambda
_{n}Ax_{n}),~n\geq 1,%
\end{array}%
\right.  \label{Al1}
\end{equation}%
where $P_{Q}:\mathcal{H}\rightarrow Q$ is the metric projection from $%
\mathcal{H}$ onto $Q$ (see Lemma \ref{L0} for the definition. They proved that if the sequence $\{(\alpha
_{n},\lambda _{n})\}$ remains in a fixed compact subset of $]0,1[\times
]0,2\nu \lbrack ,$ then any sequence $\{x_{n}\}$ generated by the process (%
\ref{Al1}) converges weakly to some element $q_{\infty }$ of $\Omega .$ To
overcome the drawback of the weak convergence and the non specification of
the limit point $q_{\infty }$ in the algorithm (\ref{Al1}) , Iiduka and
Tokahashi \cite{IT} have introduced in 2005 the following iterative process:%
\begin{equation}
\left\{
\begin{array}{l}
x_{1}\in Q \\
x_{n+1}=\alpha _{n}u+(1-\alpha _{n})SP_{Q}(x_{n}-\lambda _{n}Ax_{n}),~n\geq
1,%
\end{array}%
\right.  \label{Al2}
\end{equation}%
where $u$ is a fixed element of $Q.$ They established that if $\{\lambda
_{n}\}\in \lbrack a,b]$, with $0<a<b<2\nu $, $\{\alpha _{n}\}\in \lbrack
0,1],$ $\alpha _{n}\rightarrow 0,$ $\sum_{n\geq 1}\alpha _{n}=+\infty $ and $%
\sum_{n\geq 1}\left( \left\vert \alpha _{n+1}-\alpha _{n}\right\vert
+\left\vert \lambda _{n+1}-\lambda _{n}\right\vert \right) <\infty $, then
any sequence $\{x_{n}\}$ generated by the algorithm (\ref{Al2}) converges
strongly in $\mathcal{H}$ the closed
element of $\Omega $ to $u.$

In 2011, Yao, Liou and Chen \cite{YLC} studied two averaged version the
algorithm (\ref{AL2}). Precisely, they of introduced the following two
algorithms%
\begin{equation}
\left\{
\begin{array}[l]{l}
x_{1}\in Q \\
x_{n+1}=\beta_{n}x_{n}+(1-\beta_{n})P_{Q}(\alpha_{n}u+(1-\alpha_{n})SP_{Q}%
\left( x_{n}-\lambda_{n}Ax_{n}\right) ),~n\geq1;%
\end{array}
\right.  \label{Y1}
\end{equation}%
\begin{equation}
\left\{
\begin{array}[l]{l}
x_{1}\in Q \\
x_{n+1}=\beta_{n}x_{n}+(1-\beta_{n})SP_{Q}(\alpha_{n}u+(1-\alpha_{n})\left(
x_{n}-\lambda_{n}Ax_{n}\right) ),~n\geq1.%
\end{array}
\right.  \label{Y2}
\end{equation}
They proved that if the sequence\ $\{\alpha_{n}\}$ and $\{\lambda_{n}\}$
satisfy the same assumptions as in the Theorem of Iiduka and Tokahashi and $%
\{\beta_{n}\}$ belongs to a sub-interval $[0,b]$ of $[0,1[$ and satisfies $%
\sum_{n\geq1}\left\vert \beta_{n+1}-\beta_{n}\right\vert <\infty$, then any
sequence $\{x_{n}\}$ generated by (\ref{Y1}) or (\ref{Y2}) converges
strongly in $\mathcal{H}$ to the closed
element of $\Omega $ to $u.$

In this paper, inspired by the viscosity approximation method due to A.
Moudafi \cite{Mou}, we introduce the following iterative process%
\begin{equation}
\left\{
\begin{array}{l}
x_{1}\in Q\\
x_{n+1}=\alpha _{n}f(x_{n})+(1-\alpha _{n})SP_{Q}(x_{n}-\lambda
_{n}Ax_{n}),~n\geq 1,%
\end{array}%
\right.  \label{Al3}
\end{equation}%
which is a generalization of the algorithm (\ref{Al2}). Under the same
assumptions on the sequences $\{\lambda _{n}\}$ and $\{\alpha _{n}\}$ in the
above convergence result of Iiduka and Tokahashi, we prove that any sequence
generated by the algorithm (\ref{Al3}) converges strongly to $q^{\ast }$ the
unique solution of the variational problem%
\begin{equation}
\left\{
\begin{array}{l}
q^{\ast }\in \Omega , \\
\langle f(q^{\ast })-q^{\ast },x-q^{\ast }\rangle \leq 0,~\forall x\in
\Omega .%
\end{array}%
\right.  \tag{VP}
\end{equation}%
Moreover, we establish the strong convergence of the implicit version of the
algorithm (\ref{Al3}). Precisely, we will prove that if $\lambda
:]0,1]\rightarrow \lbrack a,b]$, with $[a,b]\subset ]0,2\nu \lbrack ,$ then
for every $t\in ]0,1]$ there exists a unique $x_{t}$ in $Q$ such that%
\begin{equation}
x_{t}=t~f(x_{t})+(1-t)~SP_{Q}(x_{t}-\lambda (t)Ax_{t}).  \label{Al4}
\end{equation}%
Then we show that $x_{t}$ converges strongly in $\mathcal{H}$ to $q^{\ast }$
as $t\rightarrow 0^{+}.$

The sequel of the paper is organized as follows: In the next section, we
recall some well-known results from convex analysis that will be useful in
the proof of the main results of the paper. In the third section, we
establish the strong convergence of the implicit algorithm (\ref{Al4}). In
the fourth section, we study the strong convergence of the explicit
algorithm (\ref{Al3}). Then we prove the stability of the process (\ref{Al4}%
) under the effect of small perturbations and we apply the obtained results
to the study of a constrained optimization problem. The last section is
devoted to the study of a numerical experiment that highlight the effect of the
sequence $\{\alpha _{n}\}$ on the convergence rate of a particular example
of the perturbed version of the algorithm (\ref{Al4}).

\section{Preliminaries}

In this section, we recall some results that will be helpful in the next
sections. Most of these results can be found in any good book on convex
analysis as \cite{BC}, \cite{Gul} and \cite{Pey}.
Let us first recall the definition of the metric projection onto a nonempty,
closed and convex subset of $\mathcal{H}$.
\begin{lemma}[{\protect\cite[Theorem 3.14]{BC}}]
\label{L0}Let $K$ be a nonempty, closed and convex subset of $\mathcal{H}$. For every $%
x\in \mathcal{H},$ there exists a unique $P_{K}(x)\in Q$ such that
\begin{equation*}
\left\Vert x-P_{K}(x)\right\Vert \leq \left\Vert x-y\right\Vert \ ~\forall
y\in K.
\end{equation*}%
The operator $P_{K}:\mathcal{H}\rightarrow K$ is called the metric
projection onto $K$.
\end{lemma}

The following classical properties of the projection operator $P_{K}$ are
very useful.

\begin{lemma}[{\protect\cite[Corollary 4.18]{BC}}]
\label{L1}Let $K$ be a nonempty, closed and convex subset of $\mathcal{H}$.

\begin{enumerate}
\item[(1)] For every $x\in \mathcal{H},$ $P_{K}(x)$ is the unique element of
$K$ which satisfies%
\begin{equation}
\langle P_{K}(x)-x,P_{K}(x)-y\rangle\leq0,\text{ for every }y\in K.
\label{h}
\end{equation}

\item[(2)] The operator $P_{K}:\mathcal{H}\rightarrow K$ is firmly
nonexpansive i.e.,%
\begin{equation}
\langle P_{K}(x)-P_{K}(y),x-y\rangle\geq\left\Vert
P_{K}(x)-P_{K}(y)\right\Vert ^{2},\text{ for all }x,y\in\mathcal{H}.
\label{h1}
\end{equation}
In particular%
\begin{equation}
\left\Vert P_{K}(x)-P_{K}(y)\right\Vert \leq\left\Vert x-y\right\Vert ,\text{
for all }x,y\in\mathcal{H}.  \label{h2}
\end{equation}
\end{enumerate}
\end{lemma}

\begin{lemma} [{\protect\cite[Theorem 3.13]{BC}}] \label{L} Let $C$ be a closed convex and nonempty
subset of $\mathcal{H}$, and $T:C\rightarrow C$ a nonexpansive mapping. Then
$F_{ix}(T)=\{x\in C:T(x)=x\}$ is a closed and convex subset of $\mathcal{H}$.
\end{lemma}

\begin{lemma}
\label{L2}Let $\lambda\in]0,2\nu].$ Then the following assertions hold true

\begin{enumerate}
\item[(i)] For every $x,y\in Q,$%
\begin{equation}
\left\Vert (x-\lambda Ax)-(y-\lambda Ay)\right\Vert ^{2}\leq\left\Vert
x-y\right\Vert ^{2}-\lambda(2\nu-\lambda)\left\Vert Ax-Ay\right\Vert ^{2}.
\label{h3}
\end{equation}

\item[(ii)] The operator $\Theta_{\lambda}:=P_{Q}\circ(I-\lambda
A):Q\rightarrow Q$ is nonexpansive and $F_{ix}(\Theta_{%
\lambda})=S_{VI(A,Q)}. $

\item[(iii)] $S_{VI(A,Q)}$ is a closed and convex subset of $\mathcal{H}.$
\end{enumerate}
\end{lemma}

\begin{proof}
(i) Let $x,y\in Q.$ A simple computation gives
\begin{align*}
\left\Vert (x-\lambda Ax)-(y-\lambda Ay)\right\Vert ^{2} & =\left\Vert
x-y\right\Vert ^{2}-2\lambda\langle Ax-Ay,x-y\rangle+\lambda^{2}\left\Vert
Ax-Ay\right\Vert ^{2} \\
& \leq\left\Vert x-y\right\Vert ^{2}-\lambda(2\nu-\lambda)\left\Vert
Ax-Ay\right\Vert ^{2}.
\end{align*}

\noindent(ii) Combining (\ref{h2}) and (\ref{h3}) yields $\Theta_{\lambda}$ is
nonexpansive. Now let $q\in Q.$ Clearly $q\in S_{VI(A,Q)}$ if and only if
\begin{equation*}
\langle q-(q-\lambda Aq),q-x\rangle\leq0~\forall x\in Q,
\end{equation*}
which, thanks to the first assertion of Lemma \ref{L1}, is equivalent to $%
q=P_{Q}(q-\lambda Aq)=\Theta_{\lambda}(q).$

\noindent The last assertion (iii) follows directly from (ii) and Lemma \ref%
{L}.
\end{proof}

The next result is a particular case of the well-known demi-closedness
principle.

\begin{lemma}[{\protect\cite[Corollary 4.18]{BC}}]
\label{L3}Let $C$ be a closed convex and nonempty subset of $\mathcal{H}$,
and $T:C\rightarrow C$ a nonexpansive mapping. If $\{x_{n}\}$ is a sequence
in $C$ weakly converging to some $\bar{x}$ such that $x_{n}-Tx_{n}$
converges strongly to $0$ in $\mathcal{H}$, then $\bar{x}\in F_{ix}(T).$
\end{lemma}

The last result of this section is a powerful lemma which is a
generalization due to Xu \cite{Xu} of a lemma proved by Berstrekas (\cite[%
Lemma 1.5.1]{Ber})

\begin{lemma}
\label{L4}let $\{a_{n}\}$ be a sequence of non negative real numbers such
that:%
\begin{equation}
a_{n+1}\leq(1-\gamma_{n})a_{n}+\gamma_{n}r_{n}+\delta_{n},\text{ }n\geq0,
\label{B}
\end{equation}
where $\{\gamma_{n}\}\in\lbrack0,1]$ and $\{r_{n}\}$and $\{\delta_{n}\}$ are
three real sequences such that:
\begin{enumerate}
\item[(1)] $\sum_{n=0}^{+\infty}\gamma_{n}=+\infty;$

\item[(2)] $\sum_{n=0}^{+\infty}\left\vert \delta_{n}\right\vert <+\infty;$

\item[(3)] $\lim\sup_{n\rightarrow+\infty}r_{n}\leq0.$
\end{enumerate}
Then the sequence $\{a_{n}\}$ converges to $0.$
\end{lemma}

\begin{proof}
We give here a proof different from the original one due to Xu \cite{Xu}.
Set $s_{n}=\sum_{k=n}^{+\infty}\delta_{k},$ $e_{n}=a_{n}+s_{n},$ and $%
\beta_{n}=\max\{r_{n},0\}+s_{n}.$ Using the fact that $%
\delta_{n}=s_{n}-s_{n+1},$ we easily obtain from (\ref{B})%
\begin{equation}
e_{n+1}\leq(1-\gamma_{n})e_{n}+\gamma_{n}\beta_{n},~n\geq0.  \label{C}
\end{equation}
Let $\varepsilon >0.$ Since $\beta_{n}\rightarrow0$ as $n\rightarrow\infty,$
there exists $n_{0}\in \mathbb{N}$ such that $\beta _{n}\leq \frac{%
\varepsilon }{2}$ for every $n\geq n_{0}.$ Let us suppose that $e_{n}\geq
\varepsilon $ for every $n\geq n_{0}.$ Hence, from (\ref{C}), we infer that,
for every $n\geq n_{0},$
\begin{eqnarray*}
e_{n}-e_{n+1} &\geq &\gamma _{n}(e_{n}-\beta_{n}) \\
&\geq &\frac{\varepsilon }{2}\gamma _{n},
\end{eqnarray*}%
which implies that $\sum_{n\geq n_{0}}\gamma_{n}<\infty .$ This is a contradiction. Then there exists $n_{1}\geq n_{0}$ such that $e_{n_{1}}\leq
\varepsilon .$ Therefore%
\begin{eqnarray*}
e_{n_{1}+1} &\leq &(1-\gamma _{n_{1}})\varepsilon +\gamma_{n_{1}}\frac{%
\varepsilon }{2} \\
&\leq &\varepsilon .
\end{eqnarray*}%
And so on we get $e_{n}\leq \varepsilon $ for every $n\geq n_{1}$. Hence $%
e_{n}\rightarrow0$ as $n\rightarrow\infty,$ which clearly implies that $%
a_{n}\rightarrow0$ as $n\rightarrow\infty$ since $s_n\rightarrow0$ as $n\rightarrow\infty$.
\end{proof}

We close this section by proving that the problem (VP) has a
unique solution.

\begin{lemma}
The problem (VP) has a unique solution $q^{\ast}$. Moreover, $q^{\ast}$ is
the unique fixed point of the contraction $P_{\Omega}\circ
f:\Omega\rightarrow\Omega$.
\end{lemma}

\begin{proof}
Let us first recall that the set $\Omega$ is nonempty, closed and convex
subset of $\mathcal{H}$. Then from the variational characterization of the
metric projection $P_{\Omega}$ (see the first assertion of Lemma \ref{L1}),
the problem (VP) is equivalent to the identity $q^{\ast}=P_{\Omega}(f(q^{%
\ast}))$. This, thanks to Banach fixed point theorem, guarantees the
existence and the uniqueness of $q^{\ast}$ since the application $%
P_{\Omega}\circ f:\Omega\rightarrow\Omega$ is clearly a contraction.
\end{proof}

\section{The convergence of the implicit algorithm (\protect\ref{Al4})}

The following section is devoted to the proof of the strong convergence of
the implicit algorithm (\ref{Al4}).

\begin{theorem}
\label{Theo1}Let $a$ and $b$ be two reals such that $0<a<b<2\nu$ and let $%
\lambda:]0,1[\rightarrow\lbrack a,b]$ be a mapping. Then, for every $%
t\in]0,1[,$ there exists a unique $x_{t}\in Q$ such that
\begin{equation*}
x_{t}=tf(x_{t})+(1-t)SP_{Q}(x_{t}-\lambda(t)Ax_{t}).
\end{equation*}
Moreover $\{x_{t}\}$ converges strongly in $\mathcal{H}$ as $%
t\rightarrow0^{+}$ to $q^{\ast}$ the unique solution of the variational
problem (VP).
\end{theorem}

The following simple lemma, which is an immediate consequence of the second
assertion of Lemma \ref{L2}, will be very useful in the proof of the
previous theorem and also in the proof of the main result of the next section.

\begin{lemma}
\label{L5}Let $t\in]0,1]$ and $\mu\in\lbrack0,2\nu].$ Then the application $%
T_{t,\mu}:Q\rightarrow Q$ defined by%
\begin{equation*}
T_{t,\mu}(x)=tf(x)+(1-t)SP_{Q}(x-\mu Ax),
\end{equation*}
satisfies%
\begin{equation*}
\left\Vert T_{t,\mu}(x)-T_{t,\mu}(y)\right\Vert \leq(1-\sigma t)\left\Vert
x-y\right\Vert ,~\forall x,y\in Q,
\end{equation*}
where $\sigma=1-\rho.$
\end{lemma}

Let us start the proof of Theorem \ref{Theo1}.

\begin{proof}
Let $t\in]0,1].$ According to Lemma \ref{L5} and the classical Banach fixed
point theorem, there exists a unique $x_{t}\in Q$ such that $%
x_{t}=T_{t,\lambda(t)}(x_{t}).$ Let us now prove that the family $%
\{x_{t}\}_{0<t\leq1}$ is bounded in $\mathcal{H}.$ Let $q\in\Omega.$ In view
of the last assertion of Lemma \ref{L2},%
\begin{align}
T_{t,\lambda(t)}(q) & =t~f(q)+(1-t)Sq  \notag \\
& =t~f(q)+(1-t)q  \label{R}
\end{align}
for every $t\in]0,1].$ Hence, Lemma \ref{L5} yields%
\begin{align*}
\left\Vert x_{t}-q\right\Vert & \leq\left\Vert
T_{t,\lambda(t)}(x_{t})-T_{t,\lambda(t)}(q)\right\Vert +t\left\Vert
f(q)-q\right\Vert \\
& \leq(1-\sigma t)\left\Vert x_{t}-q\right\Vert +t\left\Vert
f(q)-q\right\Vert ,
\end{align*}
which implies
\begin{equation*}
\sup_{0<t\leq1}\left\Vert x_{t}-q\right\Vert \leq\frac{1}{\sigma}\left\Vert
f(q)-q\right\Vert .
\end{equation*}
Hence $\{x_{t}\}_{0<t\leq1}$ is a bounded family in $\mathcal{H}.$

In the sequel, in order to simplify the notations, we will use $M$ to denote
a real constant independent of $t\in ]0,1]$ that may change from line to
another. Moreover, $\varepsilon(t)$ will simply denotes a real quantity that
converges to $0$ as the variable $t$ tends to $0.$ By the way, let us notice
here this simple result that will be often implicitly used in the sequel:
since $\{x_{t}\}_{0<t\leq1}$ is a bounded in $\mathcal{H}$, then for every
Lipschitz continuous function $g:Q\rightarrow \mathcal{H}$ the family $%
\{g(x_{t})\}_{0<t\leq1}$ is also bounded in $\mathcal{H}$.

For $t\in]0,1],$ we set $z_{t}=P_{Q}(x_{t}-\lambda(t)Ax_{t}).$ Let $q\in
\Omega.$ Clearly, by using the classical identity
\begin{equation}
\left\Vert t u+(1-t) v\right\Vert ^{2}\leq t\left\Vert u\right\Vert
^{2}+(1-t)\left\Vert v\right\Vert ^{2}, \forall u,v\in \mathcal{H},  \label{CC}
\end{equation}
the fact that $%
S $ and $P_{Q}$ are nonexpansive operators, and Lemma \ref{L2}, we obtain%
\begin{align}
\left\Vert x_{t}-q\right\Vert ^{2} & \leq t\left\Vert f(x_{t})-q\right\Vert
^{2}+(1-t)\left\Vert Sz_{t}-q\right\Vert ^{2}  \notag \\
& \leq tM+\left\Vert z_{t}-q\right\Vert ^{2}  \label{N} \\
& \leq tM+\left\Vert \left( x_{t}-\lambda(t)Ax_{t}\right) -\left(
q-\lambda(t)Aq\right) \right\Vert ^{2}  \notag \\
& \leq tM+\left\Vert x_{t}-q\right\Vert ^{2}-\lambda(t)(2\nu-\lambda
(t))\left\Vert Ax_{t}-Aq\right\Vert ^{2}.  \notag
\end{align}
We then deduce that%
\begin{equation}
a(b-2\nu)\left\Vert Ax_{t}-Aq\right\Vert ^{2}\leq tM.  \label{N1}
\end{equation}
Therefore, thanks to (\ref{h1}), we have%
\begin{align*}
\left\Vert z_{t}-q\right\Vert ^{2} & \leq\langle z_{t}-q,\left(
x_{t}-\lambda(t)Ax_{t}\right) -\left( q-\lambda(t)Aq\right) \rangle \\
& \leq\langle z_{t}-q,x_{t}-q\rangle+\lambda(t)\left\Vert z_{t}-q\right\Vert
\left\Vert Ax_{t}-Aq\right\Vert \\
& =\langle z_{t}-q,x_{t}-q\rangle+\varepsilon(t) \\
& =\frac{1}{2}\left( \left\Vert z_{t}-q\right\Vert ^{2}+\left\Vert
x_{t}-q\right\Vert ^{2}-\left\Vert x_{t}-z_{t}\right\Vert ^{2}\right)
+\varepsilon(t).
\end{align*}
The last inequality implies%
\begin{equation*}
\left\Vert z_{t}-q\right\Vert ^{2}\leq\left\Vert x_{t}-q\right\Vert
^{2}-\left\Vert x_{t}-z_{t}\right\Vert ^{2}+2\varepsilon(t).
\end{equation*}
Hence, by going back to the estimate (\ref{N}), we deduce that
\begin{equation*}
\left\Vert x_t - z_{t}\right\Vert ^{2}\leq t M+2\epsilon(t),
\end{equation*}
which implies
\begin{equation}
x_{t}-z_{t}\rightarrow0\text{ as }t\rightarrow0^{+}.  \label{N2}
\end{equation}
The last inequality in turn implies that
\begin{equation}
x_{t}-Sx_{t}\rightarrow0\text{ as }t\rightarrow0^{+}  \label{N3}
\end{equation}
Indeed,
\begin{align*}
\left\Vert x_{t}-Sx_{t}\right\Vert & \leq\left\Vert x_{t}-Sz_{t}\right\Vert
+\left\Vert Sx_{t}-Sz_{t}\right\Vert \\
& =\left\Vert T_{t,\lambda(t)}(x_{t})-Sz_{t}\right\Vert +\left\Vert
Sx_{t}-Sz_{t}\right\Vert \\
& \leq t\left\Vert f(x_{t})-Sz_{t}\right\Vert +\left\Vert
x_{t}-z_{t}\right\Vert \\
& \leq tM+\left\Vert x_{t}-z_{t}\right\Vert .
\end{align*}
Now we are in position to prove the following key result:%
\begin{equation}
\kappa:=\lim\sup_{t\rightarrow0^{+}}\langle
f(q^{\ast})-q^{\ast},x_{t}-q^{\ast}\rangle\leq0,  \label{N4}
\end{equation}
where $q^{\ast}$ is the unique solution of the variational problem (VP).
\par\noindent From the definition of $\kappa,$ there exists a sequence $\{t_{n}\}$ in $%
]0,1]$ converging to $0$ such that%
\begin{equation*}
\kappa=\lim_{n\rightarrow+\infty}\langle
f(q^{\ast})-q^{\ast},x_{t_{n}}-q^{\ast}\rangle.
\end{equation*}
On the other hand, since the family $\{x_{t_{n}}\}_{0<t\leq1}$ is a bounded
subset of the closed and convex subset $Q$ of $\mathcal{H}$, we can assume,
up to a subsequence, that $\{x_{t_{n}}\}$ converges weakly in $\mathcal{H}$
to some $x_{\infty}\in Q.$ This fact implies%
\begin{equation*}
\kappa=\langle f(q^{\ast})-q^{\ast},x_{\infty}-q^{\ast}\rangle.
\end{equation*}
Therefore, in order to prove that $\kappa\leq0,$ we just need to verify that
$x_{\infty}\in\Omega.$ Firstly, from Lemma \ref{L3} and (\ref{N3}), we have $%
x_{\infty}\in F_{ix}(S).$ Secondly, up to a subsequence, we can assume that
the real sequence $\{\lambda_{t_{n}}\}$ converges to some real $%
\lambda^{\ast }$ which belongs to $]0,2\nu\lbrack.$ Let $\Theta_{\lambda^{%
\ast}}=P_{Q}\circ(I-\lambda^{\ast}A)$ be the nonexpansive operator
introduced in Lemma \ref{L3}. Since $z_{n}=\Theta_{\lambda_{n}}(x_{t_{n}}),$
we have
\begin{align*}
\left\Vert x_{t_{n}}-\Theta_{\lambda^{\ast}}(x_{t_{n}})\right\Vert &
\leq\left\Vert x_{t_{n}}-z_{t_{n}}\right\Vert +\left\Vert
\Theta_{\lambda_{n}}(x_{t_{n}})-\Theta_{\lambda^{\ast}}(x_{t_{n}})\right\Vert
\\
& \leq\left\Vert x_{t_{n}}-z_{t_{n}}\right\Vert +\left\vert
\lambda_{t_{n}}-\lambda^{\ast}\right\vert \left\Vert Ax_{t_{n}}\right\Vert \\
& \leq\left\Vert x_{t_{n}}-z_{t_{n}}\right\Vert +\left\vert
\lambda_{t_{n}}-\lambda^{\ast}\right\vert M.
\end{align*}
Hence, by combining (\ref{N2}) and Lemma \ref{L3}, we deduce that $%
x_{\infty} $ is a fixed point of $\Theta_{\lambda^{\ast}}.$ Thus, thanks to
the second assertion of Lemma \ref{L2}, we deduce that $x_{\infty}\in VI(A,Q).$
The claim (\ref{N4}) is then proved.

Let us finally prove that $x_{t}\rightarrow q^{\ast}$ in $\mathcal{H}$ as $t$
goes to $0^{+}.$ Let $t\in]0,1].$ First, from the identity (\ref{R}), we have%
\begin{equation*}
x_{t}-q^{\ast}=u+v,
\end{equation*}
with%
\begin{align*}
u & =T_{t,\lambda(t)}(x_{t})-T_{t,\lambda(t)}\left( q^{\ast}\right) , \\
v & =t(f(q^{\ast})-q^{\ast}).
\end{align*}
Hence, by applying the inequality%
\begin{equation*}
\left\Vert u+v\right\Vert ^{2}\leq\left\Vert u\right\Vert ^{2}+2\langle
v,u+v\rangle
\end{equation*}
and Lemma \ref{L5}, we get the inequality%
\begin{equation*}
\left\Vert x_{t}-q^{\ast}\right\Vert ^{2}\leq(1-\sigma t)^{2}\left\Vert
x_{t}-q^{\ast}\right\Vert ^{2}+2t\langle f(q^{\ast})-q^{\ast},x_{t}-q^{\ast
}\rangle
\end{equation*}
which implies
\begin{align*}
\left\Vert x_{t}-q^{\ast}\right\Vert ^{2} & \leq\frac{\sigma t}{2}\left\Vert
x_{t}-q^{\ast}\right\Vert ^{2}+\frac{1}{\sigma}\langle f(q^{\ast})-q^{\ast
},x_{t}-q^{\ast}\rangle \\
& \leq tM+\frac{1}{\sigma}\langle f(q^{\ast})-q^{\ast},x_{t}-q^{\ast}\rangle.
\end{align*}
Hence, by letting $t\rightarrow0^{+}$ and using (\ref{N4}), we obtain the
desired result.
\end{proof}

\section{The convergence of the explicit algorithm (\protect\ref{Al3}).}

In this section, we study the strong convergence property of the process (\ref%
{Al3}). We prove the following theorem.

\begin{theorem}
\label{Theo2}Let $\{\alpha_{n}\}\in]0,1]$ and $\{\lambda_{n}\}\in\lbrack
0,2\nu]$ two real sequences such that:

\begin{enumerate}
\item[(i)] $\alpha_{n}\rightarrow0$ and $\sum_{n=0}^{+\infty}\alpha
_{n}=+\infty$

\item[(ii)] $0<\lim\inf_{n\rightarrow+\infty}\lambda_{n}\leq\lim
\sup_{n\rightarrow+\infty}\lambda_{n}<2\nu.$

\item[(iii)] $\frac{\alpha_{n+1}-\alpha_{n}}{\alpha_{n}}\rightarrow0$ or $%
\sum_{n=0}^{+\infty}\left\vert \alpha_{n+1}-\alpha_{n}\right\vert <\infty.$

\item[(iv)] $\frac{\lambda_{n+1}-\lambda_{n}}{\alpha_{n}}\rightarrow0$ or $%
\sum_{n=0}^{+\infty}\left\vert \lambda_{n+1}-\lambda_{n}\right\vert <\infty.$
\end{enumerate}

Then for every initial data $x_{1}\in Q$, the sequence $\{x_{n}\}$ generated
by the iterative process%
\begin{equation*}
x_{n+1}=\alpha_{n}f(x_{n})+(1-\alpha_{n})SP_{Q}(x_{n}-\lambda_{n}Ax_{n}),~n%
\geq1,
\end{equation*}
converges strongly in $\mathcal{H}$ to $q^{\ast}$ the unique solution of the
variational inequality problem (VP).
\end{theorem}

\begin{proof}
Since we are only interested on the asymptotic behavior of the sequence $%
\{x_{n}\},$ we can replace hypothesis (ii) by the stronger one: there exist
two real $a$ and $b$ in $]0,2\nu\lbrack$ such that the sequence $\{\lambda
_{n}\}$ is in $[a,b].$

For every $n\in\mathbb{N},$ we set $T_{n}:=T_{\alpha_{n},\lambda_{n}}$ where
$T_{\alpha_{n},\lambda_{n}}$ is the application defined by Lemma \ref{L5}.
First, we will prove that the sequence $\{x_{n}\}$ is bounded in $\mathcal{H}%
.$ Let $q\in\Omega.$ Thanks to the identity (\ref{R}) and Lemma \ref{L5}, we
have
\begin{align*}
\left\Vert x_{n+1}-q\right\Vert & \leq\left\Vert
T_{n}(x_{n})-T_{n}(q)\right\Vert +\alpha_n\left\Vert f(q)-q\right\Vert \\
& \leq(1-\sigma\alpha_{n})\left\Vert x_{n}-q\right\Vert +\alpha_n\left\Vert
f(q)-q\right\Vert \\
& \leq\max\{\left\Vert x_{n}-q\right\Vert ,\frac{1}{\sigma}\left\Vert
f(q)-q\right\Vert \}.
\end{align*}
Hence, we deduce by induction that
\begin{equation*}
\left\Vert x_{n}-q\right\Vert \leq\max\{\left\Vert x_{0}-q\right\Vert ,\frac{%
1}{\sigma}\left\Vert f(q)-q\right\Vert \},\forall n\in\mathbb{N}.
\end{equation*}
Therefore $\{x_{n}\}$ is bounded in $\mathcal{H}.$ Hence, for every
Lipschitz function $g:Q\rightarrow \mathcal{H},$ the sequence $\{g(x_{n})\}$
is also bounded in $\mathcal{H}.$

From hereon, as we have done in the proof of Theorem \ref{Theo1}, $M$ will
denotes a constant independent of $n$ and $\{\varepsilon_{n}\}$ a real
sequence that converges to $0.$ $M$ and $\{\varepsilon_{n}\}$ may change
from line to an other.

Let us now show that the sequence $\{\Delta x_{n}=x_{n+1}-x_{n}\}$ converges
strongly to $0.$ For every $n\in \mathbb{N},$ we clearly have
\begin{align*}
\left\Vert \Delta x_{n}\right\Vert & \leq \left\Vert
T_{n}(x_{n})-T_{n}(x_{n-1})\right\Vert +\left\Vert
T_{n}(x_{n-1})-T_{n-1}(x_{n-1})\right\Vert \\
& \leq (1-\sigma \alpha _{n})\left\Vert \Delta x_{n-1}\right\Vert +M\left[
\left\vert \Delta \alpha _{n-1}\right\vert +\left\vert \Delta \lambda
_{n-1}\right\vert \right] ,
\end{align*}%
where
\begin{equation*}
\Delta \alpha _{n}=\alpha _{n+1}-\alpha _{n},
\end{equation*}%
and
\begin{equation*}
\Delta \lambda _{n}=\lambda _{n+1}-\lambda _{n}.
\end{equation*}%
Hence, by applying Lemma \ref{L4}, we deduce that
\begin{equation}
\left\Vert \Delta x_{n}\right\Vert \rightarrow 0\text{ as }n\rightarrow
\infty .  \label{A1}
\end{equation}

For every $n\in \mathbb{N},$ we set $z_{n}=P_{Q}(x_{n}-\lambda _{n}Ax_{n}).$
Let $q\in \Omega .$ As we have proceeded in the proof of Theorem \ref{Theo1}%
, by using the classical identity (\ref{CC}) with $t=\alpha_n$, the fact that $S$ and $P_{Q}$
are nonexpansive operators, and Lemma \ref{L2}, we get
\begin{align}
\left\Vert x_{n+1}-q\right\Vert ^{2}& \leq \alpha _{n}\left\Vert
f(x_{n})-q\right\Vert ^{2}+(1-\alpha _{n})\left\Vert Sz_{n}-q\right\Vert ^{2}
\notag \\
& \leq \varepsilon _{n}+\left\Vert z_{n}-q\right\Vert ^{2}  \label{A} \\
& \leq \varepsilon _{n}+\left\Vert \left( x_{n}-\lambda _{n}Ax_{n}\right)
-\left( q-\lambda _{n}Aq\right) \right\Vert ^{2}  \notag \\
& \leq \varepsilon _{n}+\left\Vert x_{n}-q\right\Vert ^{2}-\lambda _{n}(2\nu
-\lambda _{n})\left\Vert Ax_{n}-Aq\right\Vert ^{2}.  \notag
\end{align}%
Therefore, we have%
\begin{align*}
a(2\nu -b)\left\Vert Ax_{n}-Aq\right\Vert ^{2}& \leq \varepsilon
_{n}+\left\Vert x_{n}-q\right\Vert ^{2}-\left\Vert x_{n+1}-q\right\Vert ^{2}
\\
& = \varepsilon _{n}-\langle\Delta x_{n},
x_{n+1}+x_{n}-2q\rangle \\
& \leq \varepsilon _{n}+\left\Vert \Delta x_{n}\right\Vert \left\Vert
x_{n+1}+x_{n}-2q\right\Vert \\
& \leq \varepsilon _{n}+M\left\Vert \Delta x_{n}\right\Vert .
\end{align*}%
Hence, thanks to (\ref{A1}), we deduce that%
\begin{equation}
Ax_{n}-Aq\rightarrow 0\text{ as }n\rightarrow \infty .  \label{A2}
\end{equation}%
Therefore, by using the fact that the operator $P_{Q}$ is firmly
nonexpansive (see (\ref{h1})), we get%
\begin{align*}
\left\Vert z_{n}-q\right\Vert ^{2}& \leq \langle z_{n}-q,\left(
x_{n}-\lambda _{n}Ax_{n}\right) -\left( q-\lambda _{n}Aq\right) \rangle \\
& \leq \langle z_{n}-q,x_{n}-q\rangle +\lambda _{n}\left\Vert
z_{n}-q\right\Vert \left\Vert Ax_{n}-Aq\right\Vert \\
& =\langle z_{n}-q,x_{n}-q\rangle +\varepsilon _{n} \\
& =\frac{1}{2}\left( \left\Vert z_{n}-q\right\Vert ^{2}+\left\Vert
x_{n}-q\right\Vert ^{2}-\left\Vert z_{n}-x_{n}\right\Vert ^{2}\right)
+\varepsilon _{n}.
\end{align*}%
Thus, we obtain%
\begin{equation*}
\left\Vert z_{n}-q\right\Vert ^{2}\leq \left\Vert x_{n}-q\right\Vert
^{2}-\left\Vert z_{n}-x_{n}\right\Vert ^{2}+\varepsilon _{n}.
\end{equation*}%
Inserting this inequality into (\ref{A}) yields%
\begin{align*}
\left\Vert z_{n}-x_{n}\right\Vert ^{2}& \leq \left\Vert x_{n}-q\right\Vert
^{2}-\left\Vert x_{n+1}-q\right\Vert ^{2}+\varepsilon _{n} \\
& =-\langle\Delta x_{n}, x_{n+1}+x_{n}-2q\rangle+\varepsilon _{n} \\
& \leq M\left\Vert \Delta x_{n}\right\Vert +\varepsilon _{n}.
\end{align*}%
Hence, by using (\ref{A1}), we deduce that%
\begin{equation*}
x_{n}-z_{n}\rightarrow 0\text{ as }n\rightarrow \infty .
\end{equation*}%
Therefore, by proceeding exactly as in the proof of Theorem \ref{Theo1}, we
first infer that
\begin{equation*}
x_{n}-Sx_{n}\rightarrow 0\text{ as }n\rightarrow \infty ,
\end{equation*}%
then we deduce the key result:%
\begin{equation}
\lim \sup_{n\rightarrow \infty }\langle f(q^{\ast })-q^{\ast },x_{n}-q^{\ast
}\rangle \leq 0.  \label{A3}
\end{equation}%
Let us finally prove that the sequence $\{x_{n}\}$ converges strongly in $%
\mathcal{H}$ to $q^{\ast }.$
For every $n\in \mathbb{N},$%
\begin{equation*}
x_{n+1}-q^{\ast }=T_{n}(x_{n})-T_{n}(q^{\ast })+\alpha _{n}(f(q^{\ast
})-q^{\ast }).
\end{equation*}%
Hence, by using the inequality
\begin{equation*}
\left\Vert u+v\right\Vert ^{2}\leq \left\Vert u\right\Vert ^{2}+2\langle
v,u+v\rangle ,
\end{equation*}%
with%
\begin{align*}
u& =T_{n}(x_{n})-T_{n}(q^{\ast }), \\
v& =\alpha _{n}(f(q^{\ast })-q^{\ast }),
\end{align*}%
we obtain%
\begin{align*}
\left\Vert x_{n+1}-q^{\ast }\right\Vert ^{2}& \leq \left\Vert
T_{n}(x_{n})-T_{n}(q^{\ast })\right\Vert ^{2}+2\alpha _{n}\langle f(q^{\ast
})-q^{\ast },x_{n+1}-q^{\ast }\rangle \\
& \leq (1-\sigma \alpha _{n})^{2}\left\Vert x_{n}-q^{\ast }\right\Vert
^{2}+2\alpha _{n}\langle f(q^{\ast })-q^{\ast },x_{n+1}-q^{\ast }\rangle \\
& \leq (1-2\sigma \alpha _{n})\left\Vert x_{n}-q^{\ast }\right\Vert
^{2}+\alpha _{n}\left[ 2\langle f(q^{\ast })-q^{\ast },x_{n+1}-q^{\ast
}\rangle +M\alpha _{n}\right] .
\end{align*}%
Therefore, by applying Lemma \ref{L4} and using the key result (\ref{A3}),
we deduce that the sequence $\{x_{n}\}$ converges strongly in $\mathcal{H}$
to $q^{\ast }.$ The proof is then achieved.
\end{proof}

Now we prove that the algorithm (\ref{Al3}) is stable under small
perturbations. Precisely, we establish the following result.

\begin{theorem} \label{Th42} Let $\{\alpha_{n}\}\in]0,1]$, $\{\lambda_{n}\}\in\lbrack
0,2\nu]$ and $\{e_{n}\}\in \mathcal{H}$ three sequences such that:

\begin{enumerate}
\item[(i)] $\alpha _{n}\rightarrow 0$ and $\sum_{n=0}^{+\infty }\alpha
_{n}=+\infty $

\item[(ii)] $0<\lim \inf_{n\rightarrow +\infty }\lambda _{n}\leq \lim
\sup_{n\rightarrow +\infty }\lambda _{n}<2\nu .$

\item[(iii)] $\frac{\alpha _{n+1}-\alpha _{n}}{\alpha _{n}}\rightarrow 0$ or
$\sum_{n=0}^{+\infty }\left\vert \alpha _{n+1}-\alpha _{n}\right\vert
<\infty .$

\item[(iv)] $\frac{\lambda _{n+1}-\lambda _{n}}{\alpha _{n}}\rightarrow 0$
or $\sum_{n=0}^{+\infty }\left\vert \lambda _{n+1}-\lambda _{n}\right\vert
<\infty .$

\item[(v)] $\frac{\left\Vert e_{n}\right\Vert }{\alpha _{n}}\rightarrow 0$
or $\sum_{n=0}^{+\infty }\left\Vert e_{n}\right\Vert <\infty .$
\end{enumerate}

Then for every initial data $x_{1}\in Q$, the sequence $\{x_{n}\}$ generated
by the iterative process%
\begin{equation}
x_{n+1}=P_{Q}(\alpha _{n}f(x_{n})+(1-\alpha _{n})SP_{Q}(x_{n}-\lambda
_{n}Ax_{n})+e_{n}),~n\geq 1,  \label{rr}
\end{equation}%
converges strongly in $\mathcal{H}$ to $q^{\ast }$ the unique solution of
the variational inequality problem (VP)
\end{theorem}

\begin{proof}
Let $\{y_{n}\}$ the sequence defined by%
\begin{equation*}
\left\{
\begin{array}{l}
y_{1}=x_{1}, \\
y_{n+1}=\alpha _{n}f(y_{n})+(1-\alpha _{n})SP_{Q}(y_{n}-\lambda
_{n}Ay_{n}),~n\geq 1.%
\end{array}%
\right.
\end{equation*}%
Since the sequence $\{y_{n}\}$ is in $Q,$%
\begin{eqnarray*}
\left\Vert x_{n+1}-y_{n+1}\right\Vert &=&\left\Vert
x_{n+1}-P_{Q}(y_{n+1})\right\Vert \\
&\leq &\left\Vert T_{\alpha _{n},\lambda _{n}}(x_{n})-T_{\alpha _{n},\lambda
_{n}}(y_{n})\right\Vert +\left\Vert e_{n}\right\Vert ,
\end{eqnarray*}%
where $T_{\alpha _{n},\lambda _{n}}$ is the operator defined in Lemma \ref{L5}.
Hence, for every $n\geq 1,$%
\begin{equation*}
\left\Vert x_{n+1}-y_{n+1}\right\Vert \leq (1-\sigma \alpha _{n})\left\Vert
x_{n}-y_{n}\right\Vert +\left\Vert e_{n}\right\Vert ,
\end{equation*}%
where $\sigma =1-\rho .$ Therefore, by invoking Lemma \ref{B}., we deduce that
\begin{equation*}
\left\Vert x_{n}-y_{n}\right\Vert \rightarrow 0\text{ as }n\rightarrow
\infty ,
\end{equation*}%
which implies that $\{x_{n}\}$ converges strongly in $\mathcal{H}$ to $%
q^{\ast }$ since, from Theorem \ref{Theo2}, \ the sequence $\{y_{n}\}$
converges strongly in $\mathcal{H}$ to $q^{\ast }.$
\end{proof}

As a direct consequence of Theorem \ref{Th42}, we have the following
result which improves and generalizes (\cite[Theorem 5.2]{Xu2}).

\begin{corollary}
Let $\varphi :Q\longrightarrow \mathcal{H}$ be $C^{1}$ convex function such
that its gradient $\nabla \varphi :Q\longrightarrow \mathcal{H}$ is
Lipschitz with coefficient $L>0.$ We assume that the set $F_{ix}(S)\cap \arg
\min_{Q}\varphi $ is nonempty, where $\arg \min_{Q}\varphi =\{q\in Q:\varphi
(q)\leq \varphi (x)~\forall x\in Q\}.$ Let $\{\alpha _{n}\}\in ]0,1]$, $%
\{\lambda _{n}\}\in \lbrack 0,\frac{2}{L}]$ and $\{e_{n}\}\in \mathcal{H}$
three sequences such that:

\begin{enumerate}
\item[(i)] $\alpha_{n}\rightarrow0$ and $\sum_{n=0}^{+\infty}\alpha
_{n}=+\infty$

\item[(ii)] $0<\lim\inf_{n\rightarrow+\infty}\lambda_{n}\leq\lim
\sup_{n\rightarrow+\infty}\lambda_{n}<\frac{2}{L}.$

\item[(iii)] $\frac{\alpha_{n+1}-\alpha_{n}}{\alpha_{n}}\rightarrow0$ or $%
\sum_{n=0}^{+\infty}\left\vert \alpha_{n+1}-\alpha_{n}\right\vert <\infty.$

\item[(iv)] $\frac{\lambda _{n+1}-\lambda _{n}}{\alpha _{n}}\rightarrow 0$
or $\sum_{n=0}^{+\infty }\left\vert \lambda _{n+1}-\lambda _{n}\right\vert
<\infty $

\item[(v)] $\frac{\left\Vert e_{n}\right\Vert }{\alpha _{n}}\rightarrow 0$
or $\sum_{n=0}^{+\infty }\left\Vert e_{n}\right\Vert <\infty .$
\end{enumerate}

Then for every $z_{1}\in Q$ the sequence $\{z_{n}\}$ defined iteratively by%
\begin{equation}
z_{n+1}=P_{Q}(\alpha _{n}f(z_{n})+(1-\alpha _{n})SP_{Q}(z_{n}-\lambda
_{n}\nabla \varphi (z_{n}))+e_{n})),~n\geq 1,  \label{exp}
\end{equation}%
converges strongly in $\mathcal{H}$ to $q^{\ast }$ the unique element of $%
F_{ix}(S)\cap \arg \min_{Q}\varphi $ satisfying the variational inequality%
\begin{equation}
\langle f(q^{\ast })-q^{\ast },x-q^{\ast }\rangle \leq 0  \label{sol}
\end{equation}%
for all $x\in F_{ix}(S)\cap \arg \min_{Q}\varphi .$
\end{corollary}

\begin{proof}
The proof follows directly from Theorem \ref{Th42}. In fact, according to
the famous Baillon-Haddad Theorem (\cite[Theorem 3.13]{Pey}), the operator $%
\nabla \varphi :Q\rightarrow $ $\mathcal{H}$ is $\frac{1}{L}$ inverse
strongly monotone and , from the classical varational characterization of
constrained convex problem solutions (\cite[Theorem 3.13]{Gul}), we have
\begin{equation*}
S_{VI(\nabla \varphi ,Q)}=\{q\in Q:\langle \nabla \varphi (q),x-q\rangle
\geq 0~\forall x\in Q\}=\arg \min_{Q}\varphi .
\end{equation*}
\end{proof}

\section{Numerical experiments}

In this section, we investigate through some numerical experiments the effect
of the sequence $\{\alpha _{n}\}$ on the rate convergence of sequences $%
\{z_{n}\}$ generated by a particular example of the process (\ref{exp})
studied in the previous section. Here we consider the simple case where:

\begin{enumerate}
\item[(1)] The Hilbert space $\mathcal{H}$ is $\mathbb{R}^{2}$ endowed with
its natural inner product $\langle x,y\rangle =x_{1}y_{1}+x_{2}y_{2}.$

\item[(2)] The closed and convex subset $Q$ is given by: $%
Q=\{x=(x_{1},x_{2})^{t}\in \mathbb{R}^{2}:x_{1},x_{2}\geq 0\}.$

\item[(3)] The contraction mapping  $f:Q\rightarrow Q$ is defined by $f(x)=%
\frac{1}{2}(5+\cos (x_{1}+x_{2}),6-\sin (x_{1}+x_{2}))^{t}$ for all $%
x=(x_{1},x_{2})^{t}\in Q.$ Using the mean value theorem, one can easily
verify that $f$ is Lipschitz continuous function with Lipschitz constant $%
\rho \leq \frac{\sqrt{2}}{2}.$

\item[(4)] The non expansive mapping $S:Q\rightarrow Q$ is the identity.

\item[(5)] The convex function $\varphi :Q\rightarrow \mathbb{R}$ is defined
by: $\varphi (x)=\frac{1}{2}\left\Vert Bx-b\right\Vert ^{2}$ where%
\begin{equation*}
B=\left(
\begin{array}{cc}
1 & 1 \\
2 & 2%
\end{array}%
\right) ,~b=\left(
\begin{array}{c}
3 \\
5%
\end{array}%
\right) .
\end{equation*}%
A simple calculation yields%
\begin{equation*}
\nabla \varphi (x)=B^{t}(Bx-b)=\left(
\begin{array}{c}
5x_{1}+5x_{2}-13 \\
5x_{1}+5x_{2}-13%
\end{array}%
\right) ,~\forall x=(x_{1},x_{2})^{t}\in Q.
\end{equation*}%
Hence $\nabla \varphi $ is Lipschitz continuous with Lipschitz constant $L=10.
$ Moreover,
\begin{equation*}
\Omega =F_{ix}(S)\cap \arg \min_{Q}\varphi =\{x=(x_{1},x_{2})^{t}\in
Q:x_{1}+x_{2}=2.6\}=\Delta _{2.6}^{2}
\end{equation*}%
where, for $a>0$ and $n\in \mathbb{N},$%
\begin{equation*}
\Delta _{a}^{n}=\{x=(x_{1},\cdots ,x_{n})^{t}\in \mathbb{R}^{n}:x_{1},\cdots
,x_{n}\geq 0,~x_{1}+\cdots +x_{n}=a\}.
\end{equation*}%
Let us notice that, by using KKT Theorem, one can easily verify that the
projection onto $\Delta _{a}^{n}$ is given by%
\begin{equation*}
P_{\Delta _{a}^{n}}(x)=(\max (x_{1}-\alpha (x),0),\cdots ,\max (x_{n}-\alpha
(x),0))
\end{equation*}%
for every $x=(x_{1},\cdots ,x_{n})^{t}\in \mathbb{R}^{n},$ where $\alpha (x)$
is the unique real solution $\alpha $ of the equation $\sum_{k=1}^{n}\max
(x_{k}-\alpha ,0)=a.$ Hence  a simple routine on Matlab, using the fact that
the unique solution $q^{\ast }$ to the variational problem (\ref{sol}) is
the fixed point of the contraction $P_{\Omega }\circ f:Q\rightarrow Q,$
provides a precise numerical approximation of $q^{\ast }:$%
\begin{equation*}
q^{\ast }\simeq (0.9647,1.6353)^{t}.
\end{equation*}

\item[(6)] The sequence $\{\lambda _{k}\}$ is constant and equal to $\frac{1%
}{L}=0.1.$

\item[(7)] The sequence $\{\theta _{k}\}$ is given by $\theta _{k}=\frac{1}{%
k^{\theta }}$ where $\theta $ is a constant which belongs to $]0,1].$

\item[(8)] The perturbation term $\{e_{k}\}$ is given by $e_{k}=\frac{X_{k}}{%
k^{2}}$ where $\{X_{k}\}$ is a sequence of independent random variables such
that every $X_{k}$ is uniform on the square $[-1,1]\times \lbrack -1,1].$

\item[(9)] The initial value is $z_{1}=\left(
\begin{array}{c}
2 \\
3%
\end{array}%
\right) .$

\item[(10)] $N_{\max }$ the maximal number of iterations $k$ is $N_{\max
}=6000.$
\end{enumerate}
We aim to study numerically the relation between $\theta $ and the rate
of convergence of the sequence $\{z_{k}\}$ to $q^{\ast }.$
We can summarize our numerical results in the following two points:
\par\noindent (A): The convergence of the sequence $\{z_{k}\}$ to $q^{\ast }$ is very
slow for small values of the parameter $\theta $ as the
following table shows:
\begin{table}[h]
\centering
\begin{tabular}{cc}
\hline\hline
$\theta $ & $\min_{k\leq N_{\max }}\frac{\left\Vert z_{k}-q^{\ast
}\right\Vert }{\left\Vert q^{\ast }\right\Vert }$ \\ [0.5ex]
\hline
$0.1$ &  0.4774\\
$0.2$ & 0.1810 \\
$0.3$ & 0.0742 \\
$0.4$ &  0.0309\\ [1ex]
\end{tabular}
\caption{Slow convergence of $\{z_{k}\}$ for small values of $\theta$}
\end{table}
\par\noindent (B) : The convergence of $\{z_{k}\}$ to $q^{\ast }$ is more clear if $\theta$ is close to $1$ as it is shown by the following two tables:
\begin{table}[h]
\centering
\begin{tabular}{cc}
\hline\hline
$\theta $ & $\min_{k\leq N_{\max }}\frac{\left\Vert z_{k}-q^{\ast
}\right\Vert }{\left\Vert q^{\ast }\right\Vert }$ \\[0.5ex]
\hline
$0.6$ &0.0055  \\
$0.8$ &0.0010  \\
$0.9$ &0.0005  \\
$1.0$ &0.0008\\ [1ex]
\end{tabular}
\caption{Convergence of $\{z_{k}\}$ for some values of $\theta$ closed to 1}
\end{table}
\newpage
The second table indicates, for some values of $\varepsilon >0$ and $\theta ,
$ $N(\varepsilon ,\theta )$ the first iteration $k\leq N_{\max }$ such that $%
\frac{\left\Vert z_{k}-q^{\ast }\right\Vert }{\left\Vert q^{\ast
}\right\Vert }\leq \varepsilon .$

\begin{table}[h]
\centering
\begin{tabular}{c| c c c c}
\hline\hline
$\varepsilon $& $\theta =0.6$ & $\theta =0.8$ & $\theta =0.9$ & $\theta =1.0$ \\[0.5ex]
\hline
0.5 &  6&5  &4  &4  \\
0.10 & 53 &23  &14  &17  \\
0.05 &158 &56  &36  &42  \\
0.01 &2200 &372 &249 &314 \\
0.005 &ND  &854 &533 &716 \\
0.001 &ND &ND  &2989  &4742\\[1ex]
\end{tabular}
\caption{$N(\varepsilon ,\theta )$}
\end{table}

\par\noindent Remark: $N(\varepsilon ,\theta )=$ ND (Not Defined) means that $\frac{\left\Vert
z_{k}-q^{\ast }\right\Vert }{\left\Vert q^{\ast }\right\Vert }>\varepsilon $
for all the iterations $k\leq N_{\max }.$

\par\noindent Finally, the schema (Figure \ref{convergence}) shows the convergence of $\{z_{k}\}$ to $%
q^{\ast }$ for some values of the parameter $\theta $ close to $1$.
\begin{figure}[ht]
\centering
\includegraphics[width=8cm]{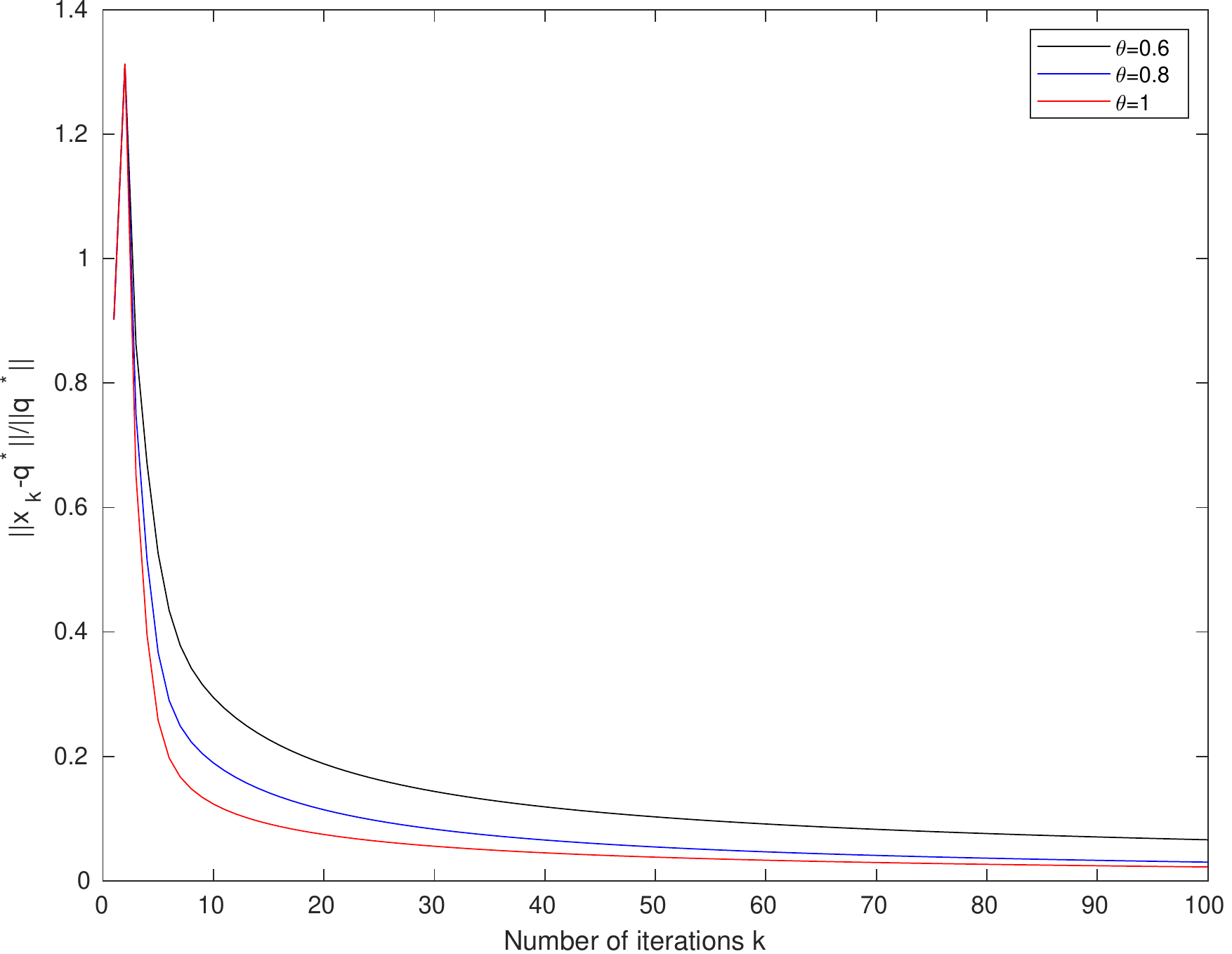}
\caption{the effect of $\theta$ on the convergence of the algorithm (\ref{exp})}
\label{convergence}
\end{figure}

\end{document}